\documentclass[12pt]{article}
\usepackage{amssymb}
\usepackage{amsmath,amsthm}
\usepackage[latin1]{inputenc}
\usepackage{color}
\usepackage{graphicx}
 \newtheorem{remark}{Remark}

 \newtheorem{theorem}[remark]{Theorem}
 
 \newtheorem{corollary}[remark]{Corollary}

\title{A note on the partition dimension of Cartesian product graphs}

\author{Ismael G. Yero and Juan A.
Rodr\'{\i}guez-Vel\'{a}zquez \\
\\
{\small Departament d'Enginyeria Inform\`{a}tica i Matem\`{a}tiques}\\
{\small Universitat Rovira i Virgili,  Av. Pa\"{\i}sos Catalans 26,
43007 Tarragona, Spain.} \\{\small ismael.gonzalez\@@urv.cat,
juanalberto.rodriguez\@@urv.cat} }

\begin{document}

\maketitle

\begin{abstract}
Let $G=(V,E)$ be a connected graph. The distance between two vertices $u,v\in V$,
denoted by $d(u, v)$, is the length of a shortest $u-v$ path in $G$. The distance between a
vertex $v\in V$ and a subset $P\subset V$ is defined as $min\{d(v, x): x \in P\}$, and
it is denoted by $d(v, P)$.
  An ordered partition $\{P_1,P_2, ...,P_t\}$ of vertices of a
graph $G$, is a \emph{resolving partition }of $G$, if all the distance
vectors $(d(v,P_1),d(v,P_2),...,d(v,P_t))$ are different. The \emph{partition dimension} of $G$, denoted by $pd(G)$, is the minimum number of
sets in any resolving partition of $G$. In this article we study the partition
dimension of Cartesian product graphs. More precisely, we show that for all pairs of connected graphs $G, H$,
$pd(G\times H)\le pd(G)+pd(H)$ and $pd(G\times H)\le pd(G)+dim(H),$  where $dim(H)$ denotes the metric dimension of $H$. Consequently, we show that $pd(G\times H)\le dim(G)+dim(H)+1.$
\end{abstract}

{\it Keywords:}  Resolving sets, resolving partition, partition
dimension, Cartesian product.

{\it AMS Subject Classification numbers:}   05C12; 05C70; 05C76
\newpage
\section{Introduction}

 The concepts of resolvability and location in graphs were described
independently by Harary and Melter \cite{harary} and Slater
\cite{leaves-trees}, to define the same structure in a
graph. After these papers were published several authors
developed diverse theoretical works about this topic
\cite{pelayo1,pelayo2,chappell,chartrand,chartrand1,chartrand2,fehr,landmarks}.
 Slater described the usefulness of these ideas into long range
aids to navigation \cite{leaves-trees}.
Also, these concepts  have some applications in chemistry for representing chemical compounds \cite{pharmacy1,pharmacy2} or to problems of
pattern recognition and image processing, some of which involve the use of hierarchical data structures \cite{Tomescu1}.
Other applications of this concept to
navigation of robots in networks and other areas appear in
\cite{chartrand,robots,landmarks}. Some variations on resolvability
or location have been appearing in the literature, like those about
conditional resolvability \cite{survey}, locating domination
\cite{haynes}, resolving domination \cite{brigham} and resolving
partitions \cite{chappell,chartrand2,fehr}.

Given a graph $G=(V,E)$ and an ordered set of vertices
$S=\{v_1,v_2,...,v_k\}$ of $G$, the {\it metric representation} of a
vertex $v\in V$ with respect to $S$ is the vector
$r(v|S)=(d(v,v_1),d(v,v_2),...,d(v,v_k))$, where $d(v,v_i)$, with
$1\leq i\leq k$, denotes the distance between the vertices $v$ and
$v_i$. We say that $S$ is a {\it resolving set} of $G$ if for every
pair of distinct vertices $u,v\in V$, $r(u|S)\ne r(v|S)$. The {\it metric
dimension}\footnote{Also called locating number.} of $G$ is the
minimum cardinality of any resolving set of $G$, and it is
denoted by $dim(G)$. The metric dimension of graphs is studied in
\cite{pelayo1,pelayo2,chappell,chartrand,chartrand1,tomescu}.

Given an ordered partition $\Pi =\{P_1,P_2, ...,P_t\}$ of the
vertices of $G$, the {\it partition representation} of a vertex
$v\in V$ with respect to the partition $\Pi$ is the vector
$r(v|\Pi)=(d(v,P_1),d(v,P_2),...,d(v,P_t))$, where $d(v,P_i)$, with
$1\leq i\leq t$, represents the distance between the vertex $v$ and
the set $P_i$, that is $d(v,P_i)=\min_{u\in P_i}\{d(v,u)\}$. We say
that $\Pi$ is a {\it resolving partition} of $G$ if for every pair
of distinct vertices $u,v\in V$, $r(u|\Pi)\ne r(v|\Pi)$. The {\it partition
dimension} of $G$ is the minimum number of sets in any resolving
partition of $G$ and it is denoted by $pd(G)$. The partition
dimension of graphs is studied in \cite{chappell,chartrand2,fehr,tomescu}.
It is natural to think that the partition dimension and metric dimension are related; in \cite{chartrand2} it was shown that for any
nontrivial connected graph $G$ we have
\begin{equation}\label{partition-dimension}pd(G)\le  dim(G) + 1.\end{equation}

The study of relationships between invariants of Cartesian product
graphs and invariants of its factors appears frequently in
research about graph theory. In the case of resolvability,  the relationships between the
metric dimension of the Cartesian product graphs and the metric
dimension of its factors was
studied in \cite{pelayo1,pelayo2}. An open problem on the dimension of Cartesian product graphs is to prove (or finding a counterexample) that for all pairs of graphs $G, H$;
$dim(G\times H)\le dim(G)+dim(H)$.
In the present paper we study the case of
resolving partition in Cartesian product graphs, by giving some
relationships between the partition dimension of  Cartesian
product graphs and the partition dimension of its factors. More precisely, we show that for all pairs of connected graphs $G, H$;
$pd(G\times H)\le pd(G)+pd(H)$ and $pd(G\times H)\le pd(G)+dim(H).$ Consequently, we show that $pd(G\times H)\le dim(G)+dim(H)+1.$

We recall that the Cartesian product of two graphs $G_1=(V_1,E_1)$ and
$G_2=(V_2,E_2)$ is the graph $G_1\times G_2=(V,E)$, such that
$V=\{(a,b)\;:\;a\in V_1,\;b\in V_2\}$ and two vertices $(a,b)\in V$
and $(c,d)\in V$ are adjacent in  $G_1\times G_2$ if and only if, either $a=c$ and $bd\in E_2$
or $b=d$ and $ac\in E_1$.

The following well known fact will be used several times.

\begin{remark}\label{distance-vertex-set}
Let the graph $G_i=(V_i,E_i)$ and let $S_i\subset V_i$, $i\in \{1,2\}$.  For every $(a,b)\in
V_1\times V_2$, it follows
$d_{G_1\times G_2}((a,b),S_1\times S_2)=d_{G_1}(a,S_1)+d_{G_2}(b,S_2).$
\end{remark}

\section{The partition dimension of Cartesian product graphs}

\begin{theorem}\label{sumapd}
For any connected graphs $G_1$ and $G_2$,
$$pd(G_1\times G_2)\leq pd(G_1)+pd(G_2).$$
\end{theorem}
\begin{proof}
Let $\Pi_1=\{A_1,A_2,...,A_k\}$ and $\Pi_2=\{B_1,B_2,...,B_t\}$ be
resolving partitions of $G_1=(V_1,E_1)$ and $G_2=(V_2,E_2)$ respectively. Let us show that
$\Pi=\{A_1\times B_1, A_1\times B_2,...,A_1\times B_t,A_2\times
B_1,A_3\times B_1,...,A_k\times B_1,C\}$, with $C=(V_1\times
V_2)-((V_1\times B_1)\cup (A_1\times V_2))$ is a resolving partition
of $G_1\times G_2$.

Let $(a,b)$, $(c,d)$ be two different vertices of $V_1\times V_2$. If $a=c$, then
there exists $B_i\in \Pi_2$ such that $d_{G_2}(b,B_i)\ne
d_{G_2}(d,B_i)$. Hence we have
\begin{align*}d_{G_1\times G_2}((a,b),A_1\times B_i)=&d_{G_1}(a,A_1)+d_{G_2}(b,B_i)\\
\ne & d_{G_1}(c,A_1)+d_{G_2}(d,B_i)\\
=& d_{G_1\times G_2}((c,d),A_1\times B_i)\end{align*}

Now, if $a\ne c$ then we have the following cases:

Case 1: Let $a\in A_i$ and $c\in A_j$, with $i\ne j$.  If we suppose,
$$d_{G_1\times G_2}((a,b),A_i\times B_1)= d_{G_1\times
G_2}((c,d),A_i\times B_1)$$
 and
$$d_{G_1\times G_2}((a,b),A_j\times
B_1)= d_{G_1\times G_2}((c,d),A_j\times B_1), $$
we obtain
\begin{align*}
d_{G_2}(b,B_1)=&d_{G_1\times G_2}((a,b),A_i\times B_1)\\
=& d_{G_1\times G_2}((c,d),A_i\times B_1)\\
=& d_{G_1}(c,A_i)+d_{G_2}(d,B_1)\\
=&  d_{G_1}(c,A_i)+d_{G_1\times G_2}((c,d),A_j\times B_1)\\
=&  d_{G_1}(c,A_i)+d_{G_1\times G_2}((a,b),A_j\times B_1) \\
=& d_{G_1}(c,A_i)+d_{G_1}(a,A_j)+d_{G_2}(b,B_1),
\end{align*}
a contradiction.

Case 2: If $a,c\in A_i$ then we have the following subcases.

Case 2.1: $b,d\in B_l$. Let $A_j\in \Pi_1$, such that
$d_{G_1}(a,A_j)\ne d_{G_1}(c,A_j)$. In this case, if $d_{G_2}(b,
B_1)=d_{G_2}(d,B_1)$ then we have
\begin{align*}d_{G_1\times G_2}((a,b),A_j\times B_1)=& d_{G_1}(a,A_j)+d_{G_2}(b,B_1)\\
                                                   \ne &d_{G_1}(c,A_j)+d_{G_2}(d,B_1)\\
                                                   =& d_{G_1\times G_2}((c,d),A_j\times B_1).
\end{align*}
On the contrary, if $d_{G_2}(b, B_1)\ne d_{G_2}(d,B_1)$ then we have
\begin{align*}d_{G_1\times G_2}((a,b),A_i\times B_1)=& d_{G_2}(b, B_1)\\
                                                      \ne & d_{G_2}(d,B_1) \\
                                                      =& d_{G_1\times G_2}((c,d),A_i\times B_1).
  \end{align*}

Case 2.2: $b\in B_j$ and $d\in B_l$, $j\ne l$. This case is analogous to Case 1.

Therefore for every different vertices $(a,b),(c,d)\in V_1\times V_2$, we have
$r((a,b)|\Pi)\ne r((c,d)|\Pi)$.
\end{proof}

By (\ref{partition-dimension}) we obtain the following direct consequence of Theorem \ref{sumapd}.
\begin{corollary}
For any connected graphs $G_1$ and $G_2$,
$$pd(G_1\times G_2)\leq pd(G_1)+dim(G_2)+1.$$
\end{corollary}

As we can see below, the above relationship can be improved.

\begin{theorem}\label{THpartDimension}
For any connected graphs $G_1$ and $G_2$, $$pd(G_1\times G_2)\le pd(G_1)+dim(G_2).$$
\end{theorem}

\begin{proof}

 Let $\Pi=\{A_1,A_2,...,A_k\}$ be a resolving
partition of $G_1=(V_1,E_1)$, let $S=\{u_1,u_2,...,u_{t}\}$ be a resolving set of
$G_2=(V_2,E_2)$ and let $C=V_1\times V_2- ((V_1\times\{u_1\})\cup (A_1\times\{u_2\})\cup \cdots \cup (A_1\times \{u_t\}))$. Let us show that
$\Pi_1=\{A_1\times \{u_1\},A_2\times \{u_1\},...,A_k\times
\{u_1\},A_1\times \{u_2\},A_1\times \{u_3\},...,A_1\times
\{u_{t}\},C\}$ is a resolving partition of $G_1\times G_2$.

Let $(a,b)$, $(c,d)$ be two different vertices of $V_1\times V_2$.
If $a=c$, then $b\ne d$. Thus, there exist $u_j\in S$ such that $d_{G_2}(b,u_j)\ne d_{G_2}(d,u_j)$. Hence,
 \begin{align*}d_{G_1\times G_2}((a,b),A_1\times \{u_j\})=&d_{G_1}(a,A_1)+d_{G_2}(b,u_j)\\
 \ne &d_{G_1}(c,A_1)+d_{G_2}(d,u_j)\\
 = &d_{G_1\times G_2}((c,d),A_1\times \{u_j\})
 \end{align*}
 Now, if $a\ne c$ we
have two cases:

Case 1: $a\in A_i$ and $c\in A_j$, $j\ne i$. Let us suppose, $d_{G_2}(b,u_1)\le d_{G_2}(d,u_1)$. In this case we have
\begin{align*}d_{G_1\times
G_2}((a,b),A_i\times\{u_1\})=&d_{G_2}(b,u_1)\\
\le & d_{G_2}(d,u_1)\\
<&d_{G_1}(c,A_i)+d_{G_2}(d,u_1)\\
=&d_{G_1\times
G_2}((c,d),A_i\times\{u_1\}).
\end{align*}
Analogously, if  $d_{G_2}(b,u_1)\ge d_{G_2}(d,u_1)$ we obtain $$d_{G_1\times
G_2}((a,b),A_j\times\{u_1\})>d_{G_1\times G_2}((c,d),A_j\times\{u_1\}).$$
Case 2: $a,c\in A_i$. Let us suppose $d_{G_2}(b,u_1)= d_{G_2}(d,u_1)$. Since there exists $j\ne i$, such that
$d_G(a,A_j)\ne d_G(c, A_j)$, we have
\begin{align*}d_{G_1\times
G_2}((a,b),A_j\times\{u_1\})=&d_{G_1}(a,A_j)+d_{G_2}(b,u_1)\\
\ne &d_{G_1}(c,A_j)+ d_{G_2}(d,u_1)\\
=&d_{G_1\times
G_2}((c,d),A_j\times\{u_1\}).
\end{align*}
If $d_{G_2}(b,u_1)\ne  d_{G_2}(d,u_1)$, we have
$d_{G_1\times
G_2}((a,b),A_i\times\{u_1\})=d_{G_2}(b,u_1)\ne d_{G_2}(d,u_1)=d_{G_1\times
G_2}((c,d),A_i\times\{u_1\}).
$
Therefore, for every different vertices
$(a,b)$, $(c,d)$ we have $r((a,b)|\Pi_1)\ne
r((c,d)|\Pi_1)$.
\end{proof}

In order to give some examples we emphasize the following well known values for the metric dimension of the complete graph, $K_n$, the path graph, $P_n$, the cycle graph, $C_n$, and the star graph, $K_{1,n}$.
\begin{remark}\label{particulares-dim}$\;$
\begin{itemize}
\item[{\rm (i)}] $dim(K_n)= n-1$ $(n\ge 2)$.
\item[{\rm (ii)}] $dim(P_n)= 1$.
\item[{\rm (iii)}]  $dim(C_n)= 2$.
\item[{\rm (iv)}]  $dim(K_{1,n})= n-1$ $(n\ge 2)$.
\end{itemize}
\end{remark}

We note that there are graphs for which Theorem \ref{sumapd}
estimates $pd(G_1\times G_2)$ better than Theorem
\ref{THpartDimension} and vice versa. For example Theorem
\ref{sumapd} leads to $pd(K_n\times P_n)\le n+2$ while Theorem
\ref{THpartDimension} gives $pd(K_n\times P_n)\le n+1$. On the
contrary, if $G$ denotes the unicyclic graph described below,
Theorem \ref{sumapd} leads to $pd(G\times G)\le 12$ while Theorem
\ref{THpartDimension} gives $pd(G\times G)\le 15$. In the above example the unicyclic graph $G$ is composed by fifteen  vertices, where the set ${1,2,3}$ form a triangle and the remaining twelve vertices are leaves: the leaves $4,5,6$ and $7$ are adjacent to $1$, the leaves $8,9,10$ and $11$ are adjacent to $2$, and the leaves $12,13,14$ and $15$ are adjacent to $3$.
In this case
$\prod=\{\{4,1,2,3\},\{8\},\{12\},\{5,9,13\},\{6,10,14\},\{7,11,15\}\}$ is
a resolving partition  and
$S=\{4,5,6,8,9,10,12,13,14\}$ is a
resolving set.

As a direct consequence of above theorem and (\ref{partition-dimension}) we deduce the following interesting result.

\begin{corollary}\label{sumadimen}
For any connected graphs $G_1$ and $G_2$, $$pd(G_1\times G_2)\le dim(G_1)+dim(G_2)+1.$$
\end{corollary}

One example of graphs for which the equality holds in Corollary \ref{sumadimen} (and also in Corollary \ref{particulares} (ii))
are the graphs belonging to the family of grid graphs: $pd(P_r\times P_t)=3$.

By Remark \ref{particulares-dim} we obtain the following particular cases of Theorem \ref{THpartDimension}.

\begin{corollary}\label{particulares}For any connected graph $G$,
\begin{itemize}
\item[{\rm (i)}] $pd(G\times K_n)\le pd(G)+n-1$.
\item[{\rm (ii)}] $pd(G\times P_n)\le pd(G)+1$.
\item[{\rm (iii)}]  $pd(G\times C_n)\le pd(G)+2$.
\item[{\rm (iv)}]  $pd(G\times K_{1,n})\le pd(G)+n-1$.
\end{itemize}
\end{corollary}

\section{Open problems}

\begin{enumerate}
\item To prove (or finding a counterexample) that for all pairs of graphs $G, H$;
$dim(G\times H)\le dim(G)+dim(H)$.
\item To provide lower bounds for $pd(G\times H)$.
\end{enumerate}

\section*{Acknowledgments}
This work was partly supported by the Spanish Ministry of Science
and Innovation through projects TSI2007-65406-C03-01 ``E-AEGIS" and
Consolider Ingenio 2010 CSD2007-0004 "ARES".

\end{document}